\documentclass[11pt]{amsart}
\usepackage{float}
\usepackage{amsfonts}
\usepackage{amsmath}
\usepackage{amssymb}
\usepackage{graphicx}
\usepackage[left=2.5cm, right=2.5cm, top=3cm]{geometry}
\usepackage{mathtools}
\usepackage{amsthm}
\usepackage{latexsym}
\usepackage{fancyhdr}
\usepackage{array}
\usepackage{amscd}
\usepackage{lscape}
\usepackage{tikz}
\usepackage{tikz-cd}
\usepackage{float}
\usepackage{bm}
\usepackage{subfigure}
\usepackage{grffile}
\usepackage{array}
\usepackage[tableposition=above]{caption}
\usepackage{longtable}
\usepackage{booktabs}
\usepackage{adjustbox}
\usepackage{diagbox}
\usepackage{lipsum}
\usepackage{makecell}
\usepackage{multirow}
\usepackage{arydshln}
\usepackage{nicematrix}
\newcolumntype{C}[1]{>{\centering\arraybackslash}p{#1}}
\usepackage[toc,page, title, titletoc]{appendix}
\definecolor{navy}{HTML}{2F729C}
\definecolor{red1}{HTML}{FF0000} % this is red
\usepackage[hyperfootnotes=false, colorlinks, linkcolor={blue}, citecolor={magenta}, filecolor={blue}, urlcolor={blue}, plainpages=false, pdfpagelabels]{hyperref}
\usepackage{algorithm}
 \usepackage{algpseudocode}
 \usepackage{algorithmicx}
 \algdef{SE}[DOWHILE]{Do}{doWhile}{\algorithmicdo}[1]{\algorithmicwhile\ #1}%

\algrenewcommand\algorithmicrequire{\textbf{Input:}}
\algrenewcommand\algorithmicensure{\textbf{Output:}}

\newcommand{\Q}{{\mathbb Q}}

\newtheorem{theorem}{Theorem}[section]
\newtheorem{lemma}[theorem]{Lemma}
\newtheorem{proposition}[theorem]{Proposition}
\newtheorem{corollary}[theorem]{Corollary}
\theoremstyle{definition}

\newtheorem{example}[theorem]{Example}

\newtheorem{mainthm}{Theorem}
\newtheorem{maincor}[mainthm]{Corollary}

\numberwithin{equation}{section}

\title{Reduced minimal models and torsion}

\author{Alexander J. Barrios}
\address{Department of Mathematics, University of St. Thomas, St. Paul, MN 55105 USA}
\email{abarrios@stthomas.edu}



\begin{document}
\begin{abstract}
Let $E/\mathbb{Q}$ be an elliptic curve. The reduced minimal model of $E$ is a global minimal model $y^{2}+a_{1}xy+a_{3}y=x^{3}+a_{2}x^{2}+a_{4}x+a_{6}$ which satisfies the additional conditions that~$a_{1},a_{3}\in~\{0,1\}$ and $a_{2}\in\{0,\pm1\}$. The reduced minimal model of $E$ is unique, and in this article, we explicitly classify the reduced minimal model of an elliptic curve $E/\mathbb{Q}$ with a non-trivial torsion point. We obtain this classification by first showing that the reduced minimal model of $E$ is uniquely determined by a congruence on $c_6$ modulo $24$. We then apply this result to parameterized families of elliptic curves to deduce our main result. We also show that the reduction at $2$ and $3$ of~$E$ affects the reduced minimal model of $E$.
\end{abstract}

\maketitle
%\setcounter{tocdepth}{1}
%\tableofcontents

\section{Introduction}
Let $E/\mathbb{Q}$ be an elliptic curve with minimal discriminant $\Delta$. Then $E$ is $\mathbb{Q}$-isomorphic to an elliptic curve given by a \textit{global minimal model} $y^{2}+a_{1}xy+a_{3}y=x^{3}+a_{2}x^{2}+a_{4}x+a_{6}$ with the property that each $a_{i}\in\mathbb{Z}$ and its discriminant is $\Delta$. The \textit{reduced minimal model} of $E$ is a
global minimal model with the property that $a_{1},a_{3}\in\left\{
0,1\right\}  $ and $a_{2}\in\left\{  0,\pm1\right\}  $. The reduced minimal
model of $E$ is unique \cite{MR1628193}. Consequently, the set of $\mathbb{Q}$-isomorphism classes of elliptic curves $E/\mathbb{Q}$ is in one-to-one correspondence with the set of elliptic curves given by their reduced minimal model. For this reason, databases of elliptic curves, such as that of LMFDB \cite{lmfdb} and Stein-Watkins \cite{MR2041090}, usually list elliptic curves $E/\mathbb{Q}$ by their reduced minimal model.

Let $y^{2}+a_{1}xy+a_{3}y=x^{3}+a_{2}x^{2}+a_{4}x+a_{6}$ denote the reduced
minimal model of $E/\mathbb{Q}$. Then there are twelve combinations for the Weierstrass coefficients
$a_{1},a_{2},$ and $a_{3}$, and we set $\operatorname{rmm}\!\left(  E\right)
=\left(  a_{1},a_{2},a_{3}\right)  $. For $1\leq i\leq12$, define
$R_{i}=\left(  a_{1},a_{2},a_{3}\right)  $ where%
\begin{equation}
{\renewcommand{\arraystretch}{1.05}\renewcommand{\arraycolsep}{.35cm}
\begin{array}
[c]{c|c|c|c|c|c|c|c|c|c|c|c|c}%
i & 1 & 2 & 3 & 4 & 5 & 6 & 7 & 8 & 9 & 10 & 11 & 12\\
a_{1} & 0 & 0 & 0 & 0 & 0 & 0 & 1 & 1 & 1 & 1 & 1 & 1\\
a_{2} & 0 & 0 & -1 & -1 & 1 & 1 & 0 & 0 & -1 & -1 & 1 & 1\\
a_{3} & 0 & 1 & 0 & 1 & 0 & 1 & 0 & 1 & 0 & 1 & 0 & 1
\end{array}}\label{ta:Ris}
\end{equation}
In this article, we show that the torsion structure of an elliptic curve $E/\mathbb{Q}$ determines the possible $\operatorname{rmm}\!\left(  E\right)  $ which can
occur. To this end, let $C_m$ denote the cyclic group of order $m$. We prove:
\begin{mainthm}
\label{mainthm}\textit{Let $T$ be one of the fifteen torsion subgroups allowed by
Mazur's Torsion Theorem~\cite{MR488287}. If $E/\mathbb{Q}$ is an elliptic curve with $T\hookrightarrow E\!\left(\mathbb{Q}\right)  _{\text{tors}}$, then $\operatorname{rmm}\!\left(  E\right)$ is one of the following $R_{i}$
for $i$ as given in the table below:}
\[
{\renewcommand{\arraystretch}{1.4}\renewcommand{\arraycolsep}{.5cm}
\begin{array}{cccccc}
T  & C_{1}& C_{2},C_{4},C_{2}\times C_{2} & C_{3} & C_{5} & C_{6}\\\hline
i & 1-12 & 1,3,5,7-12 & 1,2,5-10 & 4,6,7,12 & 1,5,7-10\\\hline\hline
T & C_{7},C_{9} & C_{8},C_{2}\times C_{4} & C_{10},C_{2}\times C_{8} &
C_{12},C_{2}\times C_{6}\\\hline
i & 7,10 & 3,5,7,12 & 7 & 7-10
\end{array}}
\]

\end{mainthm}

Now suppose that $E$ has a non-trivial torsion point. Then by Theorem~\ref{mainthm}, if $\operatorname{rmm}\!\left(  E\right)  =R_{2}$ (resp.~$R_{4}$), then $E\!\left(\mathbb{Q}\right)  _{\text{tors}}\cong C_{3}$ (resp. $C_{5}$). Since for each $R_{i}$, there exists an elliptic curve $E$ with trivial torsion subgroup such that $\operatorname{rmm}\!\left(  E\right)  =R_{i}$, the proof of Theorem~\ref{mainthm} is reduced to considering elliptic curves with a non-trivial torsion point. Parameterizations for such elliptic curves are obtained from the modular curves $X_{1}\!\left(  n\right)  $ and $X_{1}\!\left(  2,n\right)  $ \cite{MR0434947}. In this article, we consider families of elliptic curves~$E_{T}$ (see Table \ref{ta:ETmodel}) which have the property that they parameterize all rational elliptic curves with a non-trivial torsion subgroup (see Proposition \ref{rationalmodels}). Theorem \ref{mainthm} is a consequence of Theorem \ref{thm1}, which explicitly classifies $\operatorname{rmm}\!\left(E_{T}\right)  $ in terms of the parameters of $E_{T}$ (see Table~\ref{ta:rrm}). 

Given an elliptic curve $E$, a global minimal model for $E$ can be computed via Tate's algorithm~\cite{Tate1975}. Tate's algorithm also provides local information about the curve. For this reason, the algorithm needs to be run for each prime dividing the discriminant in order to obtain a global minimal model. In $1982$, Laska \cite{MR637305} gave a simpler algorithm for determining a global minimal model of an elliptic curve. In fact, the algorithm outputs the reduced minimal model of an elliptic curve. In $1989$, Kraus \cite{MR1024419} gave necessary and sufficient conditions for determining when there is an elliptic curve with Weierstrass coefficients in $\mathbb{Z}$ such that its \textit{signature} $\left(  c_{4},c_{6},\Delta\right)  $ is $\left(  \alpha,\beta,\gamma\right)  $, where $\alpha,\beta,\gamma\in\mathbb{Z}$ with $\alpha^{3}-\beta^{2}=1728\gamma\neq0$. Connell \cite{connell} then modified Laska's algorithm to make use of Kraus's theorem. The resulting algorithm is known today as the Laska-Kraus-Connell algorithm (see Algorithm~\ref{LKCA}). In Section~\ref{LKC}, we give an overview of the Laska-Kraus-Connell algorithm and show that $\operatorname{rmm}\!\left(  E\right)$ uniquely determines congruences on the $c_4$ and $c_6$ associated to a global minimal model of $E$ (see Corollary~\ref{corcong}). As a consequence, we obtain:
\begin{mainthm}
\label{thmred}\textit{Let $E/\mathbb{Q}$ be an elliptic curve. If $E$ has}
\begin{itemize}
\item[$\left(  i\right)  $] \textit{good reduction at $2$ (resp. $3$), then
$\operatorname{rmm}(E)=R_{i}$ where $i=2,4,6-12$ (resp. $i=1-12$);}

\item[$\left(  ii\right)  $] \textit{multiplicative reduction at $2$ (resp. $3$), then
$\operatorname{rmm}(E)=R_{i}$ where $i=7-12$ (resp. $i=3-8,11,12$);}

\item[$\left(  iii\right)  $] \textit{additive reduction at $2$ (resp. $3$), then
$\operatorname{rmm}(E)=R_{i}$ where $i=1,3,5$ (resp. $i=1,2,9,10$).}
\end{itemize}
\end{mainthm}

An immediate consequence of Theorem~\ref{thmred} is:
\begin{maincor}\label{coraddred}\textit{
An elliptic curve $E/\mathbb{Q}$ has additive reduction at $2$ if and only if $\operatorname{rmm}(E)=R_{i}$ where $i=1,3,5$.}
\end{maincor}

In fact, Corollary~\ref{corcong} allows us to conclude that the reduced minimal model of $E$ is uniquely determined by $c_6$ (resp. $c_6 / 2)$ modulo $24$ if $c_6$ is odd (resp. even) (see Proposition~\ref{mod24}). In Section~\ref{section3}, we explicitly classify the reduced minimal model of elliptic curves with a non-trivial torsion subgroup (see Theorem \ref{thm1}) by utilizing Proposition~\ref{mod24}. We note that the proof is computer-assisted, and only one case is done explicitly in this paper. For the remaining cases, the reader is referred to our code on GitHub \cite{GitHubRRM}, which verifies the result by exhausting all possible congruences that the parameters of $E_T$ can take modulo $24$. All coding for this article was done on SageMath~\cite{sagemath}. 

We conclude this article by considering the Cremona database \cite{cremonadata} of elliptic curves, which consists of all elliptic curves over $\mathbb{Q}$ of conductor at most $500 \ 000$. Specifically, for each of the fifteen possible torsion subgroups $T$, we compute the percentage of elliptic curves $E$ with $E(\mathbb{Q})_{\text{tors}}\cong T$ in the Cremona database that have $\operatorname{rmm}(E)=R_i$ for $1\leq i \leq 12$.

\section{Preliminaries}\label{section2}

We start by reviewing some relevant facts about elliptic curves. For further
details, see~\cite[Chapter~3]{MR1628193} and \cite{MR2514094}. Let $E/\mathbb{Q}$ be an elliptic curve given by the (affine) Weierstrass model
\begin{equation}
E:y^{2}+a_{1}xy+a_{3}y=x^{3}+a_{2}x^{2}+a_{4}x+a_{6}\label{ch:inintroweier}%
\end{equation}
with each $a_{j}\in\mathbb{Q}$. From (\ref{ch:inintroweier}), we define%
\begin{equation}%
\begin{array}
[c]{l}%
c_{4}=a_{1}^{4}+8a_{1}^{2}a_{2}-24a_{1}a_{3}+16a_{2}^{2}-48a_{4},\\
c_{6}=-\left(  a_{1}^{2}+4a_{2}\right)  ^{3}+36\left(  a_{1}^{2}%
+4a_{2}\right)  \left(  2a_{4}+a_{1}a_{3}\right)  -216\left(  a_{3}^{2}%
+4a_{6}\right)  .
\end{array}
\label{basicformulas}%
\end{equation}
The quantities $c_{4}$ and $c_{6}$ are the \textit{invariants associated to
the Weierstrass model} of $E$. The discriminant of $E$ is then defined as
$\Delta_{E}=\frac{c_{4}^{3}-c_{6}^{2}}{1728}$. We define the
\textit{signature} of $E$ to be $\operatorname*{sig}\!\left(  E\right)
=\left(  c_{4},c_{6},\Delta_{E}\right)  $. Each elliptic curve $E/\mathbb{Q}$ is $\mathbb{Q}$-isomorphic to a \textit{global minimal model} $E^{\text{min}}$ where
$E^{\text{min}}$ is given by a Weierstrass model of the form
(\ref{ch:inintroweier}) with the property that each $a_{j}\in\mathbb{Z}$ and
its discriminant $\Delta_{E}^{\text{min}}$ satisfies
\[
\Delta_{E}^{\text{min}}=\min\!\left\{  \left\vert \Delta_{F}\right\vert \mid
F\text{ is }\mathbb{Q}\text{-isomorphic to }E\text{, and }F\text{ is given by (\ref{ch:inintroweier}%
) with }a_{j}\in\mathbb{Z}\right\}  .
\]
We call $\Delta_{E}^{\text{min}}$ the \textit{minimal discriminant} of $E$.
The \textit{minimal signature} of $E$ is $\operatorname*{sig}_{\text{min}%
}(E)=\operatorname*{sig}\!\left(  E^{\text{min}}\right)  =\left(  c_{4}%
,c_{6},\Delta_{E}^{\text{min}}\right)  $, where $c_{4}$ and $c_{6}$ are the
invariants associated to a global minimal model of $E$. 
For a prime $p$, we say that $E$ has
\begin{align*}
\mathit{good\ reduction\ at\ }  & p\text{ if }p\nmid\Delta;\\
\mathit{multiplicative\ reduction\ at\ }  & p\text{ if }p|\Delta\text{ and
}p\nmid c_{4};\\
\mathit{additive\ reduction\ at\ }  & p\text{ if }p|\gcd(c_{4},\Delta).
\end{align*}

%\subsection{Elliptic curves with a non-trivial torsion point}\label{sec:torsion}

For an elliptic curve $E/\mathbb{Q}$, the Mordell-Weil group $E\!\left(\mathbb{Q}
\right)  $ is a finitely-generated abelian group. By Mazur's Torsion Theorem,
there are exactly fifteen possibilities for the torsion subgroup $E\!\left(\mathbb{Q}\right)  _{\text{tors}}$ of $E\!\left(\mathbb{Q}\right)  $:

\begin{theorem}
[Mazur's Torsion Theorem \cite{MR488287}]\label{MazurTorThm}Let $E/\mathbb{Q}$ be an elliptic curve and let $C_{m}$ denote the cyclic group of order $m$. Then
\[
E\!\left(  \mathbb{Q}\right)  _{\text{tors}}\cong\left\{
\begin{array}
[c]{ll}%
C_{m} & \text{for }m=1,2,\ldots,10,12,\\
C_{2}\times C_{2m} & \text{for }m=1,2,3,4.
\end{array}
\right.
\]
\end{theorem}

Now let $E_{T}$ be the parameterized family of elliptic curves given in Table~\ref{ta:ETmodel}
for the listed $T$. These fifteen families of elliptic curves parameterize all
elliptic curves $E/\mathbb{Q}$ with a non-trivial torsion point, as made precise by the following proposition:

\begin{proposition}
[{\cite[Proposition 4.3]{2020arXiv200101016B}}]\label{rationalmodels}Let $E/\mathbb{Q}$ be an elliptic curve and suppose further that $T\hookrightarrow E\!\left(\mathbb{Q}\right)  _{\text{tors}}$ where $T$ is one of the fourteen non-trivial torsion
subgroups allowed by Theorem~\ref{MazurTorThm}. Then there are integers
$a,b,d$ such that

$\left(  1\right)  $ If $T\neq C_{2},C_{3},C_{2}\times C_{2}$, then $E$ is
$\mathbb{Q}$-isomorphic to $E_{T}\!\left(  a,b\right)  $ with $\gcd\!\left(
a,b\right)  =1$ and $a$ is positive.

$\left(  2\right)  $ If $T=C_{2}$ and $C_{2}\times C_{2}\not \hookrightarrow
E(\mathbb{Q})$, then $E$ is $\mathbb{Q}$-isomorphic to $E_{T}\!\left(
a,b,d\right)  $ with $d\neq1,b\neq0$ such that $d$ and $\gcd\!\left(
a,b\right)  $ are positive squarefree integers.

$\left(  3\right)  $ If $T=C_{3}$ and the $j$-invariant of $E$ is not $0$,
then $E$ is $\mathbb{Q}$-isomorphic to $E_{T}\!\left(  a,b\right)  $ with
$\gcd\!\left(  a,b\right)  =1$ and $a$ is positive.

$\left(  4\right)  $ If $T=C_{3}$ and the $j$-invariant of $E$ is $0$, then
$E$ is either $\mathbb{Q}$-isomorphic to $E_{T}\!\left(  24,1\right)  $ or to
the curve $E_{C_{3}^{0}}\!\left(  a\right)  :y^{2}+ay=x^{3}$ for some positive
cubefree integer $a$.

$\left(  5\right)  $ If $T=C_{2}\times C_{2}$, then $E$ is $\mathbb{Q}%
$-isomorphic to $E_{T}\!\left(  a,b,d\right)  $ with $\gcd\!\left(
a,b\right)  =1$, $d$ positive squarefree, and $a$ is even.
\end{proposition}

\vspace{-0.85em} {\renewcommand*{\arraystretch}{1.18}
\begin{longtable}{C{0.6in}C{1in}C{1.5in}C{1.5in}C{0.5in}}
\caption[Weierstrass Model of $E_{T}$]{The Weierstrass Model of $E_{T}:y^{2}+a_{1}xy+a_{3}y=x^{3}+a_{2}x^{2}+a_{4}x$}\\
\hline
$T$ & $a_{1}$ & $a_{2}$ & $a_{3}$ & $a_{4}$ \\
\hline
\endfirsthead
\caption[]{\emph{continued}}\\
\hline
$T$ & $a_{1}$ & $a_{2}$ & $a_{3}$ & $a_{4}$\\
\hline
\endhead
\hline
\multicolumn{2}{r}{\emph{continued on next page}}
\endfoot
\hline
\endlastfoot
$C_{2}$ & $0$ & $2a$ & $0$ & $a^{2}-b^{2}d$ \\\hline
$C_{3}^{0}$ & $0$ & $0$ & $a$ & $0$ \\\hline
$C_{3}$ & $a$ & $0$ & $a^{2}b$ & $0$ \\\hline
$C_{4}$ & $a$ & $-ab$ & $-a^{2}b$ & $0$ \\\hline
$C_{5}$ & $a-b$ & $-ab$ & $-a^{2}b$ & $0$ \\\hline
$C_{6}$ & $a-b$ & $-ab-b^{2}$ & $-a^{2}b-ab^{2}$ & $0$\\\hline
$C_{7}$ & $a^{2}+ab-b^{2}$ & $a^{2}b^{2}-ab^{3}$ & $a^{4}b^{2}-a^{3}b^{3}$ & $0$ \\\hline
$C_{8}$ & $-a^{2}+4ab-2b^{2}$ & $-a^{2}b^{2}+3ab^{3}-2b^{4}$ & $-a^{3}b^{3}+3a^{2}
b^{4}-2ab^{5}$ & $0$ \\\hline
$C_{9}$ & $a^{3}+ab^{2}-b^{3}$ & $
a^{4}b^{2}-2a^{3}b^{3}+
2a^{2}b^{4}-ab^{5}
$ & $a^{3}\cdot a_{2}$
& $0$ \\\hline
$C_{10}$ &$
a^{3}-2a^{2}b-
2ab^{2}+2b^{3}
$ & $-a^{3}b^{3}+3a^{2}b^{4}-2ab^{5}$ & $(a^{3}-3a^{2}b+ab^{2})\cdot a_{2}$ & $0$\\\hline
$C_{12}$ & $
-a^{4}+2a^{3}b+2a^{2}b^{2}-
8ab^{3}+6b^{4}
$ & $b(a-2b)(a-b)^{2}(a^{2}-3ab+3b^{2})(a^{2}-2ab+2b^{2})
$ & $a(b-a)^3 \cdot a_{2} $& $0$ \\\hline
$C_{2}\times C_{2}$ & $0$ & $ad+bd$ & $0$ & $abd^{2}$ \\\hline
$C_{2}\times C_{4}$ & $a$ & $-ab-4b^{2}$ & $-a^{2}b-4ab^{2}$ & $0$ \\\hline
$C_{2}\times C_{6}$ & $-19a^{2}+2ab+b^{2}$ & $
-10a^{4}+22a^{3}b-
14a^{2}b^{2}+2ab^{3}
$ & $
90a^{6}-198a^{5}b+116a^{4}b^{2}+
4a^{3}b^{3}-14a^{2}b^{4}+2ab^{5}
$ & $0$ \\\hline
$C_{2}\times C_{8}$ & $
-a^{4}-8a^{3}b-
24a^{2}b^{2}+64b^{4}$ & $-4ab^{2}(a+2b)(a+4b)^{2}(a^{2}+4ab+8b^{2}) $ & $ -2b(a+4b)(a^{2}-8b^{2}) \cdot a_{2}
$ & $0$
\label{ta:ETmodel}	
\end{longtable}}

Next, let
\[
\left(  \alpha_{T},\beta_{T},\gamma_{T}\right)  =\left\{
\begin{array}
[c]{cl}%
\left(  \alpha_{T}\!\left(  a,b,d\right)  ,\beta_{T}\!\left(  a,b,d\right)
,\gamma_{T}\!\left(  a,b,d\right)  \right)   & \text{if }T=C_{2},C_{2}\times
C_{2},\\
\left(  \alpha_{T}\!\left(  a,b\right)  ,\beta_{T}\!\left(  a,b\right)
,\gamma_{T}\!\left(  a,b,d\right)  \right)   & \text{if }T\neq C_{2}%
,C_{2}\times C_{2}.
\end{array}
\right.
\]
be as defined in \cite[Tables 4, 5, 6]{2020arXiv200101016B}. These expressions
are also found in \cite[definitions.sage]{GitHubRRM}. By~\cite[Lemma~2.9]{2020arXiv200101016B}, $\operatorname*{sig}\!\left(  E_{T}\right)
=\left(  \alpha_{T},\beta_{T},\gamma_{T}\right)  $. Now write
\begin{equation}
a=\left\{
\begin{array}
[c]{ll}%
c^{3}d^{2}e\text{ with }d,e\text{ positive squarefree integers such that }\gcd\!\left(
d,e\right)  =1 & \text{if }T=C_{3},\\
c^{2}d\text{ with }d\text{ a squarefree integer} & \text{if }T=C_{4}.
\end{array}
\right.  \label{Cdef}
\end{equation}
Then if the parameters of $E_{T}$ satisfy the conclusion of Proposition
\ref{rationalmodels}, \cite[Theorem 4.4]{2020arXiv200101016B} gives that~$\operatorname*{sig}_{\text{min}}(E_{T})=\left(  u_{T}^{-4}\alpha_{T}%
,u_{T}^{-6}\beta_{T},u_{T}^{-12}\gamma_{T}\right)  $ where%
\[
{\renewcommand{\arraystretch}{1.2}\renewcommand{\arraycolsep}{.18cm}%
\begin{array}
[c]{cccccccc}\hline
T & C_{5},C_{7},C_{9} & C_{6},C_{8},C_{10},C_{12},C_{2}\times C_{2} &
C_{2},C_{2}\times C_{4} & C_{2}\times C_{6} & C_{2}\times C_{8} & C_{3} &
C_{4}\\\hline
u_{T} & 1 & 1\ \text{or }2 & 1,2,\ \text{or\ }4 & 1,4,\ \text{or\ }16 &
1,16,\ \text{or\ }64 & c^{2}d & c\ \text{or\ }2c\\\hline
\end{array}
}%
\]
In fact, \cite[Theorem 4.4]{2020arXiv200101016B} provides necessary and
sufficient conditions on the parameters of $E_{T}$ to determine $u_{T}$.

\section{Determining the Reduced Minimal Model from \texorpdfstring{$c_6$}{c6}}\label{LKC}

The \textit{reduced minimal model} of $E$ is a global minimal model for $E$,
which satisfies the additional property that the Weierstrass coefficients of
the model satisfy $a_{1},a_{3}\in\left\{  0,1\right\}  $ and $a_{2}\in\left\{
-1,0,1\right\}  $. The reduced minimal model of $E$ is unique, and we set
$\operatorname{rmm}\!\left(  E\right)  =\left(  a_{1},a_{2},a_{3}\right)  $.
In particular, there are twelve possibilities for $\operatorname{rmm}\!\left(
E\right)  $, and for $1\leq i\leq12$, we set $R_{i}=\left(  a_{1},a_{2}%
,a_{3}\right)  $ as given in (\ref{ta:Ris}). The reduced minimal model of $E$
is obtained from the Laska-Kraus-Connell Algorithm:

\begin{algorithm}[H]    
    \caption{\texttt{The Laska-Kraus-Connell Algorithm}}\label{LKCA}
    \begin{algorithmic}[1]
        \Require{$\operatorname{sig}_{\text{min}}(E)=(c_4,c_6,\Delta)$ for $E/\mathbb{Q}$}
        \Ensure{The reduced minimal model of $E$}
        \State Compute $b_{2}=-c_{6}\ \operatorname{mod}12\in\left\{  -5,-4,\ldots,6\right\}$
        \State Compute $b_{4}=\frac{b_{2}^{2}-c_{4}}{24}$ 
        \State Compute $b_{6}=\frac{-b_{2}^{3}+36b_{2}b_{4}-c_{6}}{216}$
        \State Compute $a_{1}=b_{2}\ \operatorname{mod}2\in\left\{  0,1\right\} $
         \State Compute $ a_{2}=\frac{b_{2}-a_{1}}{4} $
          \State Compute $ a_{3}=b_{6} \ \operatorname{mod}2\in\left\{0,1\right\}$
           \State Compute $a_{4}=\frac{b_{4}-a_{1}a_{3}}{2} $
            \State Compute $ a_{6}=\frac{b_{6}-a_{3}}{4}$
         \State \textbf{return} $y^{2}+a_{1}xy+a_{3}y=x^{3}+a_{2}x^{2}+a_{4}x+a_{6}$

    \end{algorithmic}
\end{algorithm}

We note that the original Laska-Kraus-Connell Algorithm only requires $\operatorname*{sig}\!\left(  E\right)$ for an elliptic curve $E/\mathbb{Q}$ as input (see \cite[Section
3.2]{MR1628193}). In particular, Kraus's Theorem \cite{MR1024419} is used to deduce $\operatorname{sig}_{\text{min}}(E)$ from $\operatorname*{sig}\!\left(  E\right)$.
For our purposes,
we will suppose that we have already computed $\operatorname*{sig}%
_{\text{min}}(E)$. In fact, knowledge of $\operatorname{rmm}\!\left(
E\right)  $ and $\operatorname*{sig}_{\text{min}}(E)$ determines the reduced
minimal model of $E$:

\begin{lemma}\label{lemrmm}
Let $E/\mathbb{Q}$ be an elliptic curve with $\operatorname*{sig}_{\text{min}}(E)=\left(
c_{4},c_{6},\Delta\right)  $ and $\operatorname{rmm}(E)=R_{i}$, where
$R_{i}=\left(  a_{1},a_{2},a_{3}\right)  $ is as given in \eqref{ta:Ris}. Then the
reduced minimal model of $E$ is given by
\begin{equation}
y^{2}+a_{1}xy+a_{3}y=x^{3}+a_{2}x^{2}-\frac{A_i}
{48}x-\frac{B_i}{1728},\label{rrmmodels}%
\end{equation}
where $A_i$ and $B_i$ are as given in Table~\ref{ta:RRMmodel}.
\end{lemma}
\vspace{-0.85em} {\renewcommand*{\arraystretch}{1.24}
\begin{longtable}{C{.75in}C{0.3in}C{0.3in}C{0.3in}C{1.0in}C{1.5in}}
\caption{The reduced minimal model of $E$, $y^{2}+a_{1}xy+a_{3}y=x^{3}+a_{2}x^{2}-\frac{A_i}{48}x-\frac{B_i}{1728}$, in terms of $R_i$ and  $\operatorname*{sig}_{\text{min}}(E)=\left(
c_{4},c_{6},\Delta\right)  $}\\
\hline
$\operatorname{rmm}(E)$ & $a_{1}$ & $a_{2}$ & $a_{3}$ & $A_i$ & $B_i$ \\
\hline
\endfirsthead
\caption[]{\emph{continued}}\\
\hline
$\operatorname{rmm}(E)$ & $a_{1}$ & $a_{2}$ & $a_{3}$ & $A$ & $B$\\
\hline
\endhead
\hline
\multicolumn{6}{r}{\emph{continued on next page}}
\endfoot
\hline
\endlastfoot
$R_{1}$ & $0$ & $0$ & $0$ & $c_{4}$ & $2c_{6}$\\\hline
$R_{2}$ & $0$ & $0$ & $1$ & $c_{4}$ & $2(c_{6}+216)$\\\hline
$R_{3}$ & $0$ & $-1$ & $0$ & $c_{4}-16$ & $2(  -6c_{4}+c_{6}+32)
$\\\hline
$R_{4}$ & $0$ & $-1$ & $1$ & $c_{4}-16$ & $2(  -6c_{4}+c_{6}+248)
$\\\hline
$R_{5}$ & $0$ & $1$ & $0$ & $c_{4}-16$ & $2(  6c_{4}+c_{6}-32)  $\\\hline
$R_{6}$ & $0$ & $1$ & $1$ & $c_{4}-16$ & $2(  6c_{4}+c_{6}+184)
$\\\hline
$R_{7}$ & $1$ & $0$ & $0$ & $c_{4}-1$ & $3c_{4}+2c_{6}-1$\\\hline
$R_{8}$ & $1$ & $0$ & $1$ & $c_{4}+23$ & $3c_{4}+2c_{6}+431$\\\hline
$R_{9}$ & $1$ & $-1$ & $0$ & $c_{4}-9$ & $-9c_{4}+2c_{6}+27$\\\hline
$R_{10}$ & $1$ & $-1$ & $1$ & $c_{4}+15$ & $-9c_{4}+2c_{6}+459$\\\hline
$R_{11}$ & $1$ & $1$ & $0$ & $c_{4}-25$ & $15c_{4}+2c_{6}-125$\\\hline
$R_{12}$ & $1$ & $1$ & $1$ & $c_{4}-1$ & $15c_{4}+2c_{6}+307$%
\label{ta:RRMmodel}	
\end{longtable}}

\begin{proof}
Let $\operatorname{rmm}\!\left(  E\right)  =R_{i}$. For $1\leq i\leq12$, let
$F_{i}:y^{2}+a_{1}xy+a_{3}y=x^{3}+a_{2}x^{2}+a_{4}x+a_{6}$ be an elliptic
curve over $\mathbb{Q}\!\left(  a_{4},a_{6}\right)  $. Computing the invariants $c_{4}$ and $c_{6}$
of $F_{i}$ yields
\[
c_{4}=\left\{
\begin{array}
[c]{ll}%
-48a_{4} & \text{if }i=1\\
-48a_{4} & \text{if }i=2\\
-16\left(  3a_{4}-1\right)   & \text{if }i=3\\
-16\left(  3a_{4}-1\right)   & \text{if }i=4\\
-16\left(  3a_{4}-1\right)   & \text{if }i=5\\
-16\left(  3a_{4}-1\right)   & \text{if }i=6\\
-\left(  48a_{4}-1\right)   & \text{if }i=7\\
-\left(  48a_{4}+23\right)   & \text{if }i=8\\
-3\left(  16a_{4}-3\right)   & \text{if }i=9\\
-3\left(  16a_{4}+5\right)   & \text{if }i=10\\
-\left(  48a_{4}-25\right)   & \text{if }i=11\\
-\left(  48a_{4}-1\right)   & \text{if }i=12
\end{array}
\right. \quad  \text{ and } \quad c_{6}=\left\{
\begin{array}
[c]{ll}%
-864a_{6} & \text{if }i=1\\
-216\left(  4a_{6}+1\right)   & \text{if }i=2\\
-32\left(  9a_{4}+27a_{6}-2\right)   & \text{if }i=3\\
-8\left(  36a_{4}+108a_{6}+19\right)   & \text{if }i=4\\
-32\left(  -9a_{4}+27a_{6}+2\right)   & \text{if }i=5\\
-8\left(  -36a_{4}+108a_{6}+35\right)   & \text{if }i=6\\
-\left(  -72a_{4}+864a_{6}+1\right)   & \text{if }i=7\\
-\left(  -72a_{4}+864a_{6}+181\right)   & \text{if }i=8\\
-27\left(  8a_{4}+32a_{6}-1\right)   & \text{if }i=9\\
-27\left(  8a_{4}+32a_{6}+11\right)   & \text{if }i=10\\
-\left(  -360a_{4}+864a_{6}+125\right)   & \text{if }i=11\\
-\left(  -360a_{4}+864a_{6}+161\right)   & \text{if }i=12
\end{array}
\right.
\]
For each $i$, solving for $a_{4}$ and $a_{6}$ in terms of $c_{4}$ and $c_{6}$
allows us to verify that $a_{4}=-\frac{A_i}{48}$ and $a_{6}=-\frac{B_i}{1728}$ for $A_i$ and $B_i$ as given in Table~\ref{ta:RRMmodel} in terms
of $c_{4}$ and $c_{6}$. This result was verified on SageMath \cite{sagemath},
and the verification is found in \cite[Section3.ipynb]{GitHubRRM}.
\end{proof}

As a result, given an elliptic curve $E$ with invariants $c_{4}$ and $c_{6}$
associated to a global minimal model of $E$, the reduced minimal model is
uniquely determined upon computing $\operatorname{rmm}(E)$.

\begin{corollary}
\label{corcong}Let $E/\mathbb{Q}$ be an elliptic curve with $\operatorname*{sig}_{\text{min}}(E)=\left(
c_{4},c_{6},\Delta\right)  $ and $\operatorname*{rmm}(E)=R_i  $ as given in (\ref{ta:Ris}). Then $c_{4}$ and $c_{6}$
satisfy the congruences given below:

\begin{equation}
{\renewcommand{\arraystretch}{1.2}\renewcommand{\arraycolsep}{.5cm}
\begin{array}
[c]{ccccccc}\cline{1-3}\cline{5-7}%
i & c_{4} & c_{6} & \qquad & i & c_{4} & c_{6}\\\cline{1-3}%
\cline{5-7}
1 & \multicolumn{1}{r}{0\ \operatorname{mod}48} &
\multicolumn{1}{r}{0\ \operatorname{mod}\ 864} &  & 7 &
1\ \operatorname{mod}48 & 71\ \operatorname{mod}72\\\cline{1-3}\cline{5-7}%
2 & \multicolumn{1}{r}{0\ \operatorname{mod}48} &
\multicolumn{1}{r}{648\ \operatorname{mod}\ 864} &  & 8 &
25\ \operatorname{mod}48 & 35\ \operatorname{mod}72\\\cline{1-3}%
\cline{5-7}%
3 & \multicolumn{1}{r}{16\ \operatorname{mod}48} &
\multicolumn{1}{r}{64\ \operatorname{mod}\ 288} &  & 9 &
9\ \operatorname{mod}48 & 27\ \operatorname{mod}72\\\cline{1-3}\cline{5-7}%
4 & \multicolumn{1}{r}{16\ \operatorname{mod}48} &
\multicolumn{1}{r}{136\ \operatorname{mod}288} &  & 10 &
33\ \operatorname{mod}48 & 63\ \operatorname{mod}72\\\cline{1-3}%
\cline{5-7}%
5 & \multicolumn{1}{r}{16\ \operatorname{mod}48} &
\multicolumn{1}{r}{224\ \operatorname{mod}288} &  & 11 &
25\ \operatorname{mod}48 & 19\ \operatorname{mod}72\\\cline{1-3}%
\cline{5-7}%
6 & \multicolumn{1}{r}{16\ \operatorname{mod}48} &
\multicolumn{1}{r}{8\ \operatorname{mod}288} &  & 12 &
1\ \operatorname{mod}48 & 55\ \operatorname{mod}72\\\cline{1-3}\cline{5-7}
\end{array}}\label{cong}
\end{equation}

\end{corollary}

\begin{proof}
For each $i\in\left\{  1,\ldots,12\right\}  $, let $A_i$ and $B_i$ be as given in
Table \ref{ta:RRMmodel} in terms of $c_{4}$ and $c_{6}$. By Lemma
\ref{lemrmm}, $A_i\equiv0\ \operatorname{mod}48$ and $B_i\equiv
0\ \operatorname{mod}1728$. Solving for $c_{4}$ in $A_i$ modulo $48$ yields the
claimed congruences in (\ref{cong}). Next, solving for $2c_{6}$ in $B_i$ modulo
$1728$ allows us to determine $c_{6}$ modulo $864$ with the established
congruences for~$c_{4}$. It is then verified that the congruences modulo $864$
for~$c_{6}$ reduce to the claimed congruences in (\ref{cong}). This result was
verified on SageMath \cite{sagemath}, and the verification is found in
\cite[Section3.ipynb]{GitHubRRM}.
\end{proof}

With this result, we are now ready to prove Theorem \ref{thmred}:

\begin{proof}
[Proof of Theorem \ref{thmred}.]Let $\operatorname*{sig}_{\text{min}%
}(E)=\left(  c_{4},c_{6},\Delta\right)  $. By Corollary \ref{corcong},
$\operatorname*{rmm}(E)=R_{i}$ for $1\leq i\leq12$ uniquely determines
congruences on $c_{4}$ and $c_{6}$. In particular, we have that the $2$-adic and
$3$-adic valuations of $c_{4}$ and $c_{6}$ are as given below:

\[
{\renewcommand{\arraystretch}{1.2}\renewcommand{\arraycolsep}{.35cm}
\begin{array}
[c]{ccccccc}\cline{1-3}\cline{5-7}
i & \left(  v_{2}(c_{4}),v_{2}(c_{6})\right)   & \left(  v_{3}%
(c_{4}),v_{3}(c_{6})\right)   & \qquad & i & \left(  v_{2}(c_{4}%
),v_{2}(c_{6})\right)   & \left(  v_{3}(c_{4}),v_{3}(c_{6})\right)
\\\cline{1-3}\cline{5-7}%
1 & \left(  \geq4,\geq5\right)   & \left(  \geq1,\geq3\right)   &  & 7
& \left(  0,0\right)   & \left(  0,0\right)  \\\cline{1-3}\cline{5-7}%
2 & \left(  \geq4,3\right)   & \left(  \geq1,\geq3\right)   &  & 8 &
\left(  0,0\right)   & \left(  0,0\right)  \\\cline{1-3}\cline{5-7}%
3 & \left(  \geq4,\geq5\right)   & \left(  0,0\right)   &  & 9 &
\left(  0,0\right)   & \left(  \geq1,\geq2\right)  \\\cline{1-3}%
\cline{5-7}%
4 & \left(  \geq4,3\right)   & \left(  0,0\right)   &  & 10 & \left(
0,0\right)   & \left(  \geq1,\geq2\right)  \\\cline{1-3}\cline{5-7}%
5 & \left(  \geq4,\geq5\right)   & \left(  0,0\right)   &  & 11 &
\left(  0,0\right)   & \left(  0,0\right)  \\\cline{1-3}\cline{5-7}%
6 & \left(  \geq4,3\right)   & \left(  0,0\right)   &  & 12 & \left(
0,0\right)   & \left(  0,0\right)  \\\cline{1-3}\cline{5-7}
\end{array}}
\]
The result now follows from \cite[Tableau II and Tableau IV]{MR1225948}.
\end{proof}

The next result establishes that the reduced minimal model is uniquely determined by a congruence depending on $c_6$ modulo $24$:

\begin{proposition}
\label{mod24}Let $E/%
%TCIMACRO{\U{211a} }%
%BeginExpansion
\mathbb{Q}
%EndExpansion
$ be an elliptic curve with $\operatorname*{sig}_{\text{min}}(E)=\left(
c_{4},c_{6},\Delta\right)  $. Let $a_{1}=c_{6}\ \operatorname{mod}2\in\left\{
0,1\right\}  $. Then $\operatorname*{rmm}(E)=R_{i}$ if%
\begin{equation}\label{mod24eq}
{\renewcommand{\arraystretch}{1.2}\renewcommand{\arraycolsep}{.3cm}
\begin{array}
[c]{ccccccccccccc}%
i & 1 & 2 & 3 & 4 & 5 & 6 & 7 & 8 & 9 & 10 & 11 &
12\\\hline
2^{a_{1}-1}c_{6}\ \operatorname{mod}24 & \multicolumn{1}{r}{0} &
\multicolumn{1}{r}{12} & 8 & 20 & 16 & 4 & 23 & 11 & 3 & 15
& 19 & 7\\\hline
\end{array}}
\end{equation}
In particular, if $A_{i}$ and $B_{i}$ are as defined in Table~\ref{ta:RRMmodel} , then the reduced minimal model of $E$ is%
\begin{equation}
y^{2}+a_{1}xy+a_{3}y=x^{3}+a_{2}x^{2}-\frac{A_i}
{48}x-\frac{B_i}{1728}.\label{rrmmodels1}%
\end{equation}

\end{proposition}
\begin{proof}
From Corollary \ref{corcong}, we have that $a_{1}=0$ if and only if $2^{a_1-1}c_{6}$ is
even. Moreover, reducing the congruences for $c_{6}$ in (\ref{cong}) modulo $24$ yields the congruences listed in \eqref{mod24eq}. The result now follows by Lemma~\ref{lemrmm}.
\end{proof}

\begin{example}
As a demonstration of Proposition~\ref{mod24}, we consider the elliptic curve $E:y^{2}=x^{3}-11346507x+16371897606$ (\href{http://www.lmfdb.org/EllipticCurve/Q/1830/l/1}{LMFDB label
1830.l1}). By the first part of the Laska-Kraus-Connell Algorithm \cite[Section
3.2]{MR1628193}, we find that
\[
\operatorname*{sig}\nolimits_{\text{min}}(E)=\left(
420241,-303183289,-10245657600000\right)  .
\]
Since $c_{6}\equiv23\ \operatorname{mod}24$, we have by Proposition
\ref{mod24} that $\operatorname*{rmm}(E)=R_{7}$ and the reduced minimal model
of $E$ is given by%
\begin{align*}
y^{2}+xy  & =x^{3}-\frac{c_{4}-1}{48}x-\frac{3c_{4}+2c_{6}-1}{1728}\\
& =x^{3}-8755x+350177.
\end{align*}

\end{example}

\section{Classification of Reduced Minimal Models}\label{section3}

In this section, we obtain Theorem \ref{mainthm} as a consequence of our
explicit classification of the reduced minimal model of $E_{T}$. By Proposition~\ref{mod24}, the computation of the reduced minimal model is reduced to
computing $\operatorname*{sig}_{\text{min}}(E_{T})$ and $\operatorname{rmm}%
(E_{T})$. By \cite[Theorem 4.4]{2020arXiv200101016B}, there are necessary and
sufficient conditions on the parameters of $E_{T}$ to obtain
$\operatorname*{sig}_{\text{min}}(E_{T})=\left(  u_{T}^{-4}\alpha_{T}%
,u_{T}^{-6}\beta_{T},u_{T}^{-12}\gamma_{T}\right)  $. Theorem \ref{thm1} gives
necessary and sufficient conditions on the parameters of $E_{T}$ to determine
$\operatorname{rmm}\!\left(  E_{T}\right)  $:

\begin{theorem}\label{thm1}
Let $E_{T}$ be as given in Table \ref{ta:ETmodel}. Suppose that the parameters
of $E_{T}$ satisfy the conclusion of Proposition \ref{rationalmodels}, and let
$a=c^{2}d$ for $d$ a positive squarefree integer if $T=C_{4}$. Then there are
necessary and sufficient conditions on the parameters of $E_{T}$ to determine
the reduced minimal model of $E_{T}$. Table \ref{ta:rrm} summarizes these
necessary and sufficient conditions.
\end{theorem}

{\begingroup  \footnotesize %\tiny
\renewcommand{\arraystretch}{1.3}
 \begin{longtable}{ccccc}
 	\caption{The reduced minimal model of $E_T$}\\
	\hline
$T$ & $\operatorname{rmm}(E_{T})$ & \multicolumn{3}{c}{Conditions on parameters} \\
	\hline
	\endfirsthead
	\hline
$T$ & $\operatorname{rmm}(E_{T})$ & \multicolumn{3}{c}{Conditions on parameters} \\
	\hline
	\endhead
	\hline 
	\multicolumn{5}{r}{\emph{continued on next page}}
	\endfoot
	\hline 
	\endlastfoot
$C_{2}$ & $R_{1}$ & $a\equiv0\ \operatorname{mod}3$ & $v_{2}\!\left(
b\right)  \leq2$ or $a\not \equiv 3\ \operatorname{mod}4$ & $v_{2}(b^{2}d-a^{2})\leq3$ or $v_{2}(  a)  \neq1$\\\cmidrule(lr){3-5}
&  & $a\equiv0\ \operatorname{mod}6$ & $b\equiv2\ \operatorname{mod}4$ &
$v_{2}(b^{2}d-a^{2})\leq7$ or $a\not \equiv 2\ \operatorname{mod}8$\\\cmidrule(lr){2-5}
& $R_{3}$ & $a\equiv1\ \operatorname{mod}3$ & $v_{2}\!\left(  b\right)  \leq2$
or $a\not \equiv 3\ \operatorname{mod}4$ & $v_{2}(b^{2}d-a^{2})\leq3$ or $v_{2}(  a)  \neq1$\\\cmidrule(lr){3-5}
&  & $a\equiv4\ \operatorname{mod}6$ & $b\equiv2\ \operatorname{mod}4$ &
$v_{2}(b^{2}d-a^{2})\leq7$ or $a\not \equiv 2\ \operatorname{mod}8$\\\cmidrule(lr){2-5}
& $R_{5}$ & $a\equiv2\ \operatorname{mod}3$ & $v_{2}\!\left(  b\right)  \leq2$
or $a\not \equiv 3\ \operatorname{mod}4$ & $v_{2}(b^{2}d-a^{2})\leq3$ or $v_{2}(  a)  \neq1$\\\cmidrule(lr){3-5}
&  & $a\equiv2\ \operatorname{mod}6$ & $b\equiv2\ \operatorname{mod}4$ &
$v_{2}(b^{2}d-a^{2})\leq7$ or $a\not \equiv 2\ \operatorname{mod}8$\\\cmidrule(lr){2-5}
& $R_{7}$ & $a\equiv2\ \operatorname{mod}48$ & $b\equiv2\ \operatorname{mod}4$
& $v_{2}(b^{2}d-a^{2})\geq8$\\\cmidrule(lr){3-5}
&  & $a\equiv23\ \operatorname{mod}24$ & $b\equiv0\ \operatorname{mod}8$ & \\\cmidrule(lr){2-5}
& $R_{8}$ & $a\equiv26\ \operatorname{mod}48$ & $b\equiv2\ \operatorname{mod}%
4$ & $v_{2}(b^{2}d-a^{2})\geq8$\\\cmidrule(lr){3-5}
&  & $a\equiv11\ \operatorname{mod}24$ & $b\equiv0\ \operatorname{mod}8$ & \\\cmidrule(lr){2-5}
& $R_{9}$ & $a\equiv42\ \operatorname{mod}48$ & $b\equiv2\ \operatorname{mod}%
4$ & $v_{2}(b^{2}d-a^{2})\geq8$\\\cmidrule(lr){3-5}
&  & $a\equiv3\ \operatorname{mod}24$ & $b\equiv0\ \operatorname{mod}8$ & \\\cmidrule(lr){2-5}
& $R_{10}$ & $a\equiv18\ \operatorname{mod}48$ & $b\equiv2\ \operatorname{mod}%
4$ & $v_{2}(b^{2}d-a^{2})\geq8$\\\cmidrule(lr){3-5}
&  & $a\equiv15\ \operatorname{mod}24$ & $b\equiv0\ \operatorname{mod}8$ & \\\cmidrule(lr){2-5}
& $R_{11}$ & $a\equiv10\ \operatorname{mod}48$ & $b\equiv2\ \operatorname{mod}%
4$ & $v_{2}(b^{2}d-a^{2})\geq8$\\\cmidrule(lr){3-5}
&  & $a\equiv19\ \operatorname{mod}24$ & $b\equiv0\ \operatorname{mod}8$ & \\\cmidrule(lr){2-5}
& $R_{12}$ & $a\equiv34\ \operatorname{mod}48$ & $b\equiv2\ \operatorname{mod}%
4$ & $v_{2}(b^{2}d-a^{2})\geq8$\\\cmidrule(lr){3-5}
&  & $a\equiv7\ \operatorname{mod}24$ & $b\equiv0\ \operatorname{mod}8$ & \\\hline
$C_{3}$ & $R_{1}$ & $a\equiv0\ \operatorname{mod}6$ & $v_{2}\!\left(
a\right)  \not \equiv 0\ \operatorname{mod}3$ & \\\cmidrule(lr){2-5}
& $R_{2}$ & $a\equiv0\ \operatorname{mod}6$ & $v_{2}\!\left(  a\right)
\equiv0\ \operatorname{mod}3$ & \\\cmidrule(lr){2-5}
& $R_{5}$ & $a\equiv\pm2\ \operatorname{mod}6$ & $v_{2}\!\left(  a\right)
\not \equiv 0\ \operatorname{mod}3$ & \\\cmidrule(lr){2-5}
& $R_{6}$ & $a\equiv\pm2\ \operatorname{mod}6$ & $v_{2}\!\left(  a\right)
\equiv0\ \operatorname{mod}3$ & \\\cmidrule(lr){2-5}
& $R_{7}$ & $a\equiv\pm1\ \operatorname{mod}6$ & $b$ is even & \\\cmidrule(lr){2-5}
& $R_{8}$ & $a\equiv\pm1\ \operatorname{mod}6$ & $b$ is odd & \\\cmidrule(lr){2-5}
& $R_{9}$ & $a\equiv3\ \operatorname{mod}6$ & $b$ is odd & \\\cmidrule(lr){2-5}
& $R_{10}$ & $a\equiv3\ \operatorname{mod}6$ & $b$ is even & \\\hline
$C_{3}^{0}$ & $R_{1}$ & $a$ is even &  & \\\cmidrule(lr){2-5}
& $R_{2}$ & $a$ is odd &  & \\\hline
$C_{4}$ & $R_{1}$ & $v_{2}(a)\leq7$ or $bd\not \equiv 3\ \operatorname{mod}4$
& $a$ is even & $ab(a+b)\not \equiv 0\ \operatorname{mod}3$ or $v_{3}(a)$ is
odd\\\cmidrule(lr){2-5}
& $R_{3}$ & $v_{2}(a)\leq7$ or $bd\not \equiv 3\ \operatorname{mod}4$ & $a$ is
even & $a+b\equiv0\ \operatorname{mod}3$\\\cmidrule(lr){5-5}
&  &  &  & $v_{3}(a)>0$ is even and $bd\equiv1,4\ \operatorname{mod}6$\\\cmidrule(lr){2-5}
& $R_{5}$ & $v_{2}(a)\leq7$ or $bd\not \equiv 3\ \operatorname{mod}4$ & $a$ is
even & $b\equiv0\ \operatorname{mod}3$\\\cmidrule(lr){5-5}
&  &  &  & $v_{3}(a)>0$ is even and $bd\equiv2,5\ \operatorname{mod}6$\\\cmidrule(lr){2-5}
& $R_{7}$ & $v_{2}(a)\leq7$ or $bd\not \equiv 3\ \operatorname{mod}4$ & $a$ is
odd & $b\equiv0\ \operatorname{mod}3$\\\cmidrule(lr){5-5}
&  &  &  & $v_{3}(a)>0$ is even and $bd\equiv2,5\ \operatorname{mod}6$\\\cmidrule(lr){3-5}
&  & $v_{2}(a)\geq8$ is even & $bd\equiv7,15\ \operatorname{mod}16$ &
$b\equiv0\ \operatorname{mod}3$\\\cmidrule(lr){5-5}
&  &  &  & $v_{3}(a)>0$ is even and $bd\equiv11\ \operatorname{mod}12$\\\cmidrule(lr){2-5}
& $R_{8}$ & $v_{2}(a)\geq8$ is even & $bd\equiv3,11\ \operatorname{mod}16$ &
$b\equiv0\ \operatorname{mod}3$\\\cmidrule(lr){5-5}
&  &  &  & $v_{3}(a)>0$ is even and $bd\equiv11\ \operatorname{mod}12$\\\cmidrule(lr){2-5}
& $R_{9}$ & $v_{2}(a)\geq8$ is even & $bd\equiv3,11\ \operatorname{mod}16$ &
$ab(a+b)\not \equiv 0\ \operatorname{mod}3$ or $v_{3}(a)$ is odd\\\cmidrule(lr){2-5}
& $R_{10}$ & $v_{2}(a)\leq7$ or $bd\not \equiv 3\ \operatorname{mod}4$ & $a$
is odd & $ab(a+b)\not \equiv 0\ \operatorname{mod}3$ or $v_{3}(a)$ is odd\\\cmidrule(lr){3-5}
&  & $v_{2}(a)\geq8$ is even & $bd\equiv7,15\ \operatorname{mod}16$ &
$ab(a+b)\not \equiv 0\ \operatorname{mod}3$ or $v_{3}(a)$ is odd\\\cmidrule(lr){2-5}
& $R_{11}$ & $v_{2}(a)\geq8$ is even & $bd\equiv3,11\ \operatorname{mod}16$ &
$a+b\equiv0\ \operatorname{mod}3$\\\cmidrule(lr){5-5}
&  &  &  & $v_{3}(a)>0$ is even and $bd\equiv7\ \operatorname{mod}12$\\\cmidrule(lr){2-5}
& $R_{12}$ & $v_{2}(a)\leq7$ or $bd\not \equiv 3\ \operatorname{mod}4$ & $a$
is odd & $a+b\equiv0\ \operatorname{mod}3$\\\cmidrule(lr){5-5}
&  &  &  & $v_{3}(a)>0$ is even and $bd\equiv1,4\ \operatorname{mod}6$\\\cmidrule(lr){3-5}
&  & $v_{2}(a)\geq8$ is even & $bd\equiv7,15\ \operatorname{mod}16$ &
$a+b\equiv0\ \operatorname{mod}3$\\\cmidrule(lr){5-5}
&  &  &  & $a\equiv0\ \operatorname{mod}3$ and $bd\equiv7\ \operatorname{mod}%
12$\\\hline
$C_{5}$ & $R_{4}$ & $ab\equiv\pm1\ \operatorname{mod}6$ &  & \\\cmidrule(lr){2-5}
& $R_{6}$ & $ab\equiv3\ \operatorname{mod}6$ &  & \\\cmidrule(lr){2-5}
& $R_{7}$ & $ab\equiv0\ \operatorname{mod}6$ &  & \\\cmidrule(lr){2-5}
& $R_{12}$ & $ab\equiv\pm2\ \operatorname{mod}6$ &  & \\\hline
$C_{6}$ & $R_{1}$ & $a\equiv3\ \operatorname{mod}6$ & $v_{2}(a+b)=1,2$ & \\\cmidrule(lr){2-5}
& $R_{5}$ & $a\equiv\pm1\ \operatorname{mod}6$ & $v_{2}(a+b)=1,2$ & \\\cmidrule(lr){2-5}
& $R_{7}$ & $a\equiv\pm1\ \operatorname{mod}6$ & $v_{2}(a+b)\neq1,2,3$ & \\\cmidrule(lr){2-5}
& $R_{8}$ & $a\equiv\pm1\ \operatorname{mod}6$ & $v_{2}(a+b)=3$ & \\\cmidrule(lr){3-5}
&  & $a\equiv\pm2\ \operatorname{mod}6$ &  & \\\cmidrule(lr){2-5}
& $R_{9}$ & $a\equiv3\ \operatorname{mod}6$ & $v_{2}(a+b)=3$ & \\\cmidrule(lr){3-5}
&  & $a\equiv0\ \operatorname{mod}6$ &  & \\\cmidrule(lr){2-5}
& $R_{10}$ & $a\equiv3\ \operatorname{mod}6$ & $v_{2}(a+b)\neq1,2,3$ & \\\hline
$C_{7}$ & $R_{7}$ & $a+b\equiv\pm1\ \operatorname{mod}3$ &  & \\\cmidrule(lr){2-5}
& $R_{10}$ & $a+b\equiv0\ \operatorname{mod}3$ &  & \\\hline
$C_{8}$ & $R_{3}$ & $a\equiv0\ \operatorname{mod}12$ &  & \\\cmidrule(lr){2-5}
& $R_{5}$ & $a\equiv\pm4\ \operatorname{mod}12$ &  & \\\cmidrule(lr){2-5}
& $R_{7}$ & $a\equiv\pm1,\pm2,\pm5\ \operatorname{mod}12$ &  & \\\cmidrule(lr){2-5}
& $R_{12}$ & $a\equiv\pm3,6\ \operatorname{mod}12$ &  & \\\hline
$C_{9}$ & $R_{7}$ & $a+b\equiv\pm1\ \operatorname{mod}3$ &  & \\\cmidrule(lr){2-5}
& $R_{10}$ & $a+b\equiv0\ \operatorname{mod}3$ &  & \\\hline
$C_{10}$ & $R_{7}$ & $v_{2}(a)\geq0$ &  & \\\hline
$C_{12}$ & $R_{7}$ & $a\equiv\pm1,\pm2,\pm5\ \operatorname{mod}12$ &  & \\\cmidrule(lr){2-5}
& $R_{8}$ & $a\equiv\pm4\ \operatorname{mod}12$ &  & \\\cmidrule(lr){2-5}
& $R_{9}$ & $a\equiv0\ \operatorname{mod}12$ &  & \\\cmidrule(lr){2-5}
& $R_{10}$ & $a\equiv\pm3,6\ \operatorname{mod}12$ &  & \\\hline
\multicolumn{2}{l}{$C_{2}\times C_{2}\hspace{.5em}   R_{1}$} & %\multicolumn{2}{l}{
$v_{2}(a)\leq3$ or $bd\not \equiv 1\ \operatorname{mod}4$%} 
& $d(a+b)\equiv
0\ \operatorname{mod}3$\\\cmidrule(lr){2-5}
& $R_{3}$ & %\multicolumn{2}{l}{
$v_{2}(a)\leq3$ or $bd\not \equiv
1\ \operatorname{mod}4$%} 
& $d(a+b)\equiv2\ \operatorname{mod}3$\\\cmidrule(lr){2-5}
& $R_{5}$ & %\multicolumn{2}{l}{
$v_{2}(a)\leq3$ or $bd\not \equiv
1\ \operatorname{mod}4$%} 
& $d(a+b)\equiv1\ \operatorname{mod}3$\\\cmidrule(lr){2-5}
& $R_{7}$ & $v_{2}(a)\geq4$ & $bd\equiv1\ \operatorname{mod}4$ &
$d(a+b)\equiv1\ \operatorname{mod}24$\\\cmidrule(lr){2-5}
& $R_{8}$ & $v_{2}(a)\geq4$ & $bd\equiv1\ \operatorname{mod}4$ &
$d(a+b)\equiv13\ \operatorname{mod}24$\\\cmidrule(lr){2-5}
& $R_{9}$ & $v_{2}(a)\geq4$ & $bd\equiv1\ \operatorname{mod}4$ &
$d(a+b)\equiv21\ \operatorname{mod}24$\\\cmidrule(lr){2-5}
& $R_{10}$ & $v_{2}(a)\geq4$ & $bd\equiv1\ \operatorname{mod}4$ &
$d(a+b)\equiv9\ \operatorname{mod}24$\\\cmidrule(lr){2-5}
& $R_{11}$ & $v_{2}(a)\geq4$ & $bd\equiv1\ \operatorname{mod}4$ &
$d(a+b)\equiv5\ \operatorname{mod}24$\\\cmidrule(lr){2-5}
& $R_{12}$ & $v_{2}(a)\geq4$ & $bd\equiv1\ \operatorname{mod}4$ &
$d(a+b)\equiv17\ \operatorname{mod}24$\\\hline
\multicolumn{2}{l}{$C_{2}\times C_{4}\hspace{.5em} R_{3}$} & $a\equiv
6\ \operatorname{mod}12$ & \multicolumn{2}{l}{$ab\equiv6\ \operatorname{mod}%
12$}\\\cmidrule(lr){3-5}
&  & $a\equiv0\ \operatorname{mod}12$ & \multicolumn{2}{l}{$ab\equiv
0,24,36\ \operatorname{mod}48$}\\\cmidrule(lr){3-5}
&  & $a\equiv\pm2\ \operatorname{mod}12$ & \multicolumn{2}{l}{$ab\equiv
10\ \operatorname{mod}12$}\\\cmidrule(lr){3-5}
&  & $a\equiv\pm4\ \operatorname{mod}12$ & \multicolumn{2}{l}{$ab\equiv
4,16,40\ \operatorname{mod}48$}\\\cmidrule(lr){2-5}
& $R_{5}$ & $a\equiv\pm2\ \operatorname{mod}12$ & \multicolumn{2}{l}{$ab\equiv
2,6\ \operatorname{mod}12$}\\\cmidrule(lr){3-5}
&  & $a\equiv\pm4\ \operatorname{mod}12$ & \multicolumn{2}{l}{$ab\equiv
0,8,20,24,32,36\ \operatorname{mod}48$}\\\cmidrule(lr){2-5}
& $R_{7}$ & $a\equiv\pm1\ \operatorname{mod}6$ & \multicolumn{2}{l}{$ab\equiv
0,2\ \operatorname{mod}3$}\\\cmidrule(lr){3-5}
&  & $a\equiv\pm4\ \operatorname{mod}12$ & \multicolumn{2}{l}{$ab\equiv
12,44\ \operatorname{mod}48$}\\\cmidrule(lr){2-5}
& $R_{12}$ & $a\equiv\pm1\ \operatorname{mod}6$ & \multicolumn{2}{l}{$ab\equiv
1\ \operatorname{mod}3$}\\\cmidrule(lr){3-5}
&  & $a\equiv3\ \operatorname{mod}6$ & \multicolumn{2}{l}{$ab\equiv
0\ \operatorname{mod}3$}\\\cmidrule(lr){3-5}
&  & $a\equiv\pm4\ \operatorname{mod}12$ & \multicolumn{2}{l}{$ab\equiv
28\ \operatorname{mod}48$}\\\cmidrule(lr){3-5}
&  & $a\equiv0\ \operatorname{mod}12$ & \multicolumn{2}{l}{$ab\equiv
12\ \operatorname{mod}48$}\\\hline
\multicolumn{2}{l}{$C_{2}\times C_{6}\hspace{.5em} R_{7}$} & $a+b$ is odd &
\multicolumn{2}{l}{$b\not \equiv 0\ \operatorname{mod}3$}\\\cmidrule(lr){3-5}
&  & $a+b$ is even & \multicolumn{2}{l}{$a(a+b)\equiv
2,6,18,38\ \operatorname{mod}48$}\\\cmidrule(lr){2-5}
& $R_{8}$ & $a+b$ is even & \multicolumn{2}{l}{$a(a+b)\equiv
0,8,12,14,20,24,26,30,32,36,42,44\ \operatorname{mod}48$}\\\cmidrule(lr){2-5}
& $R_{9}$ & $a+b$ is even & \multicolumn{2}{l}{$a(a+b)\equiv
4,10,16,28,40,46\ \operatorname{mod}48$}\\\cmidrule(lr){2-5}
& $R_{10}$ & $a+b$ is odd & \multicolumn{2}{l}{$b\equiv0\ \operatorname{mod}%
3$}\\\cmidrule(lr){3-5}
&  & $a+b$ is even & \multicolumn{2}{l}{$a(a+b)\equiv22,34\ \operatorname{mod}%
48$}\\\hline
\multicolumn{2}{l}{$C_{2}\times C_{8}\hspace{.5em} R_{7}$} & $v_{2}(a)\geq0$ &  &
\label{ta:rrm}
\end{longtable}
\endgroup}

\begin{proof}
The proof of this result is done by considering each $E_{T}$ separately. We
observe that for each~$T$, the given conditions on the parameters in Table
\ref{ta:rrm}  to obtain $R_{i}$ partition the integers $a,b,d$ that satisfy
the assumptions in the conclusion to Proposition \ref{rationalmodels}. For
each $T$, we also have necessary and sufficient conditions on the parameters
of $E_{T}$ to obtain $\operatorname*{sig}_{\text{min}}(E_{T})=\left(
u_{T}^{-4}\alpha_{T},u_{T}^{-6}\beta_{T},u_{T}^{-12}\gamma_{T}\right)  $. By
Proposition \ref{mod24} it suffices to compute $\operatorname*{rmm}(E_{T})$ by
considering $u_{T}^{-6}\beta_{T}$ or $u_{T}^{-6}\beta_{T}/2$ modulo $24$. In particular, it suffices to
exhaust all possible congruence classes on the parameters of $E_{T}$ modulo~$24$ to deduce $\operatorname*{rmm}(E_{T})$. Since the method of proof is the
same in each case, we only provide a proof for the $T=C_{2}\times C_{2}$ case
in this article. The proof has been automated for all the cases, and its
verification is found in \cite[Section4.ipynb]{GitHubRRM}.

Suppose $T=C_{2}\times C_{2}$ and that the parameters of $E_{T}$ satisfy the
following conditions: $a,b,d$ are integers with $a$ even, $\gcd\left(
a,b\right)  =1$, and $d>0$ is squarefree. By \cite[Theorem 4.4]%
{2020arXiv200101016B}, $\operatorname*{sig}_{\text{min}}(E_{T})=\left(
c_{4},c_{6},\Delta\right)  =\left(  u_{T}^{-4}\alpha_{T},u_{T}^{-6}\beta
_{T},u_{T}^{-12}\gamma_{T}\right)  $ where%
\[
u_{T}=\left\{
\begin{array}
[c]{cl}%
1 & \text{if }v_{2}(a)\leq3\text{ or }bd\not \equiv 1\ \operatorname{mod}4,\\
2 & \text{if }v_{2}(a)\geq4\text{ and }bd\equiv1\ \operatorname{mod}4.
\end{array}
\right.
\]
In particular,%
\begin{equation}
c_{6}=\left\{
\begin{array}
[c]{cl}%
-32d^{3}\left(  2a-b\right)  \left(  a+b\right)  \left(  a-2b\right)   &
\text{if }u_{T}=1\\
-d^{3}\left(  2a-b\right)  \left(  a+b\right)  \left(  \frac{a}{2}-b\right)
& \text{if }u_{T}=2.
\end{array}
\right.  \label{C2C2pf}%
\end{equation}
This is verified in \cite[detailedC2C2.ipynb]{GitHubRRM}, and the
statements below are also verified in that file.

\textbf{Case 1.} Let $v_{2}(a)\leq3$ or $bd\not \equiv 1\ \operatorname{mod}%
4$. Then $c_6$ is even and the
claim is verified in this case by Proposition~\ref{mod24}, since
\[
\frac{c_{6}}{2}\equiv16d^{3}\left(  a+b\right)  ^{3}\ \operatorname{mod}%
24=\left\{
\begin{array}
[c]{rl}%
0\ \operatorname{mod}24 & \text{if }d\left(  a+b\right)  \equiv
0\ \operatorname{mod}3,\\
8\ \operatorname{mod}24 & \text{if }d\left(  a+b\right)  \equiv
2\ \operatorname{mod}3,\\
16\ \operatorname{mod}24 & \text{if }d\left(  a+b\right)  \equiv
1\ \operatorname{mod}3.
\end{array}
\right.
\]

\textbf{Case 2. }Let $v_{2}\!\left(  a\right)  \geq4$ and $bd\equiv
1\ \operatorname{mod}4$. Then $c_6$ is odd and the result now follows for $T=C_{2}\times C_{2}$ by
Proposition~\ref{mod24} since%
\[
c_{6}\equiv-d^{3}\left(  a+b\right)  ^{3}\ \operatorname{mod}24=\left\{
\begin{array}
[c]{rl}%
23\ \operatorname{mod}24 & \text{if }d\left(  a+b\right)  \equiv
1\ \operatorname{mod}24,\\
11\ \operatorname{mod}24 & \text{if }d\left(  a+b\right)  \equiv
13\ \operatorname{mod}24,\\
3\ \operatorname{mod}24 & \text{if }d\left(  a+b\right)  \equiv
21\ \operatorname{mod}24,\\
15\ \operatorname{mod}24 & \text{if }d\left(  a+b\right)  \equiv
9\ \operatorname{mod}24,\\
19\ \operatorname{mod}24 & \text{if }d\left(  a+b\right)  \equiv
5\ \operatorname{mod}24,\\
7\ \operatorname{mod}24 & \text{if }d\left(  a+b\right)  \equiv
17\ \operatorname{mod}24.
\end{array}
\right.
\]

As noted, the remaining cases are verified in \cite[Section4.ipynb]%
{GitHubRRM}. While it suffices to exhaust all congruence classes on the
parameters modulo $24$, special care must be taken for those $T$ where conditions on the parameters leads to
$u_{T}>1$. Indeed, in the proof above, we observe that when $u_{T}=2$, we have
an $\frac{a}{2}$ appearing in the expression of $c_{6}$. The assumptions that
$v_{2}(a)\geq4$ yields that the possible values of $a$ modulo $24$ are
$0,8,16$. Reducing $\frac{a}{2}$ modulo $24$ results in the same congruences
classes. However, if instead the assumption had been $v_{2}(a)=1$, we would
have needed to consider $a$ modulo $48$ to ensure that we do exhaust all possible
congruence classes for $\frac{a}{2} \ \operatorname{mod} 24$. Our code takes this into account for the remaining $T$'s
where this occurs.
\end{proof}

By Corollary~\ref{coraddred}, an elliptic curve $E/%
%TCIMACRO{\U{211a} }%
%BeginExpansion
\mathbb{Q}
%EndExpansion
$ has additive reduction at $p=2$ if and only if $\operatorname*{rmm}%
(E)=R_{i}$, where $i=1,3,5$. In particular, the cases corresponding to
$\operatorname*{rmm}(E_{T})=R_{i}$ for $i=1,3,5$ are precisely the cases for
which $E_{T}$ has additive reduction at $2$. In \cite{BarRoy}, necessary and
sufficient conditions on the parameters of $E_{T}$ were given to deduce the
local data of $E_{T}$ at primes for which $E_{T}$ has additive reduction. A
comparison of loc. cit. with Theorem \ref{thm1} shows that
$\operatorname*{rmm}(E)$ does not encode any further information about the
local data at $p=2$.

Next, we use Theorem~\ref{thm1} and Proposition~\ref{mod24} to compute the reduced minimal models of the elliptic curves appearing in Examples~8.5 and 8.6 of \cite{2020arXiv200101016B}.

\begin{example}
\label{example1}The elliptic curve%
\[
E:y^{2}=x^{3}-1900650154752x+990015042347311104
\]
is $%
%TCIMACRO{\U{211a} }%
%BeginExpansion
\mathbb{Q}
%EndExpansion
$-isomorphic to $E_{C_{4}}\!\left(  a,b\right)  $ where $\left(  a,b\right)
=\left(  2^{12}\cdot3^{2},5\cdot7\cdot131\right)  $. In particular, $d=1$ in
the notation of (\ref{Cdef}). It follows from Theorem \ref{thm1} that
$\operatorname*{rmm}(E)=R_{3}$ since $v_{3}(a)=2$ and $bd\equiv
1\ \operatorname{mod}6$. By Proposition \ref{mod24}, the reduced minimal model
of $E$ is%
\begin{align*}
y^{2}  & =x^{3}-x^{2}-\frac{c_{4}-16}{48}-\frac{2(c_{6}-6c_{4}+32)}{1728}\\
& =x^{3}-x^{2}-91659440x+331584587712.
\end{align*}
For the last step, we have that the invariants $c_{4}$ and $c_{6}$ associated
to a global minimal model of $E$ are $c_{4}=4399653136$ and $c_{6}%
=-286462685864384$.
\end{example}

\begin{example}
\label{example2}The elliptic curve%
\[
E:y^{2}=x^{3}-19057987954261048752x+31955359661403338940204703104
\]
is $%
%TCIMACRO{\U{211a} }%
%BeginExpansion
\mathbb{Q}
%EndExpansion
$-isomorphic to $E_{C_{12}}\!\left(  6,11\right)  $. From Theorem \ref{thm1}
we deduce that $\operatorname*{rmm}(E)=R_{10}$. The reduced minimal model is
then obtained from Proposition \ref{mod24}:%
\begin{align*}
y^{2}+xy+y  & =x^{3}-x^{2}-\frac{c_{4}+15}{48}x-\frac{2c_{6}-9c_{4}+459}%
{1728}\\
& =x^{3}-x^{2}-919077351189287x+10701785524467279561311.
\end{align*}
We note that $c_{4}$ and $c_{6}$ are $44115712857085761$ and
$-9246342494619021684087009$, respectively.
\end{example}

We conclude by considering the Cremona database \cite{cremonadata}, which currently consists of all elliptic
curves $E/\mathbb{Q}$ whose conductor is at most $500\ 000$. This amounts to a total of $3\ 064\ 705$ elliptic curves. Below, we give the number $n_{T}$ of elliptic curves in the Cremona database with torsion subgroup $T$:
\[
{\renewcommand{\arraystretch}{1.2}\renewcommand{\arraycolsep}{.25cm}
\begin{array}
[c]{cccccccccccccc}
T & n_{T} &  & T & n_{T} &  & T & n_{T} &  & T & n_{T} &  &
T & n_{T}\\\cline{1-2}\cline{4-5}\cline{7-8}\cline{10-11}\cline{13-14}%
C_{1} & 1683021 &  & C_{4} & 33558 &  & C_{7} & 80 &  & C_{10} &
42 &  & C_{2}\times C_{4} & 1737\\\cline{1-2}\cline{4-5}\cline{7-8}%
\cline{10-11}\cline{13-14}%
C_{2} & 1186350 &  & C_{5} & 1503 &  & C_{8} & 178 &  & C_{12} &
17 &  & C_{2}\times C_{6} & 96\\\cline{1-2}\cline{4-5}\cline{7-8}%
\cline{10-11}\cline{13-14}%
C_{3} & 51405 &  & C_{6} & 6759 &  & C_{9} & 20 &  & C_{2}\times
C_{2} & 99933 &  & C_{2}\times C_{8} & 6\\\cline{1-2}\cline{4-5}%
\cline{7-8}\cline{10-11}\cline{13-14}
\end{array}}
\]
Table \ref{dist} gives the distribution of $\operatorname*{rmm}(E)$ among the
$n_{T}$ elliptic curves with specified torsion subgroup $T$ in the Cremona database. The code used to compute the data in the table is found in \cite[Cremonadatabase.ipynb]{GitHubRRM}.

\vspace{1em}

{\begingroup \footnotesize
\renewcommand{\arraystretch}{1.15}
\captionof{table}{Distribution of $\operatorname{rmm}(E)$ for elliptic curves $E$ with $E(\Q)_{\text{tors}}\cong T$ and conductor $<500\ 000$}
\begin{center}
\begin{tabular}
[c]{ccccccccccccc}
\diagbox[width=4em]{\vspace{-0.3em}$T$}{\vspace{-0.8em}$R_i$} & \vspace{-1.3em} $R_{1}$ & $R_{2}$ & $R_{3}$ & $R_{4}$ & $R_{5}$ & $R_{6}$ & $R_{7}$ &
$R_{8}$ & $R_{9}$ & $R_{10}$ & $R_{11}$ & $R_{12}$\\\\\hline
$C_{1}$ & $17.0\%$ & $5.54\%$ & $11.7\%$ & $3.63\%$ & $11.3\%$ & $3.73\%$ &
$6.85\%$ & $6.71\%$ & $10.1\%$ & $10.1\%$ & $6.67\%$ & $6.72\%$\\\hline
$C_{2}$ & $18.5\%$ & $0\%$ & $14.4\%$ & $0\%$ & $14.3\%$ & $0\%$ & $7.84\%$ &
$8.10\%$ & $10.5\%$ & $10.2\%$ & $8.11\%$ & $7.97\%$\\\hline
$C_{3}$ & $7.52\%$ & $7.67\%$ & $0\%$ & $0\%$ & $8.79\%$ & $9.29\%$ & $16.9\%$
& $19.7\%$ & $14.4\%$ & $15.7\%$ & $0\%$ & $0\%$\\\hline
$C_{4}$ & $12.9\%$ & $0\%$ & $15.3\%$ & $0\%$ & $15.7\%$ & $0\%$ & $14.9\%$ &
$3.89\%$ & $2.99\%$ & $13.3\%$ & $3.94\%$ & $17.0\%$\\\hline
$C_{5}$ & $0\%$ & $0\%$ & $0\%$ & $10.8\%$ & $0\%$ & $16.6\%$ & $39.0\%$ &
$0\%$ & $0\%$ & $0\%$ & $0\%$ & $33.6\%$\\\hline
$C_{6}$ & $5.33\%$ & $0\%$ & $0\%$ & $0\%$ & $8.73\%$ & $0\%$ & $24.1\%$ &
$28.4\%$ & $15.7\%$ & $17.8\%$ & $0\%$ & $0\%$\\\hline
$C_{7}$ & $0\%$ & $0\%$ & $0\%$ & $0\%$ & $0\%$ & $0\%$ & $73.8\%$ & $0.0\%$ &
$0.0\%$ & $26.3\%$ & $0\%$ & $0\%$\\\hline
$C_{8}$ & $0\%$ & $0\%$ & $4.49\%$ & $0\%$ & $12.9\%$ & $0\%$ & $59.0\%$ &
$0\%$ & $0\%$ & $0\%$ & $0\%$ & $23.6\%$\\\hline
$C_{9}$ & $0\%$ & $0\%$ & $0\%$ & $0\%$ & $0\%$ & $0\%$ & $75.0\%$ & $0.0\%$ &
$0.0\%$ & $25.0\%$ & $0\%$ & $0\%$\\\hline
$C_{10}$ & $0\%$ & $0\%$ & $0\%$ & $0\%$ & $0\%$ & $0\%$ & $100\%$ & $0\%$ &
$0\%$ & $0\%$ & $0\%$ & $0\%$\\\hline
$C_{12}$ & $0\%$ & $0\%$ & $0\%$ & $0\%$ & $0\%$ & $0\%$ & $41.2\%$ & $23.5\%$
& $0.0\%$ & $0.0\%$ & $0\%$ & $0\%$\\\hline
$C_{2}\times C_{2}$ & $17.8\%$ & $0\%$ & $13.5\%$ & $0\%$ & $13.6\%$ & $0\%$ &
$8.52\%$ & $7.91\%$ & $11.3\%$ & $10.6\%$ & $7.89\%$ & $8.89\%$\\\hline
$C_{2}\times C_{4}$ & $0\%$ & $0\%$ & $17.8\%$ & $0\%$ & $18.6\%$ & $0\%$ &
$29.6\%$ & $0\%$ & $0\%$ & $0\%$ & $0\%$ & $34.0\%$\\\hline
$C_{2}\times C_{6}$ & $0\%$ & $0\%$ & $0\%$ & $0\%$ & $0\%$ & $0\%$ & $25.0\%$
& $32.3\%$ & $17.7\%$ & $25.0\%$ & $0\%$ & $0\%$\\\hline
$C_{2}\times C_{8}$ & $0\%$ & $0\%$ & $0\%$ & $0\%$ & $0\%$ & $0\%$ & $100\%$
& $0\%$ & $0\%$ & $0\%$ & $0\%$ & $0\%$\\\hline \label{dist}
\end{tabular}
\end{center}
\endgroup}

\noindent \textbf{Acknowledgments.} The author would like to thank Alyson Deines, Enrique Gonz\'{a}lez-Jim\'{e}nez, Daniel Ortega, and Manami Roy for helpful conversation as the article was being written.  In particular, their python suggestions helped simplify the code verifying Theorem~\ref{thm1}.

\bibliographystyle{amsplain}
\bibliography{bibliography}
\end{document}